\documentclass[a4paper]{amsart}
\usepackage[T1]{fontenc}
\usepackage[english]{babel}
\usepackage[cp1252]{inputenc}
\usepackage{amsthm}
\usepackage{amsmath}
\usepackage{amsfonts}
\usepackage{stmaryrd}
\usepackage{amssymb}
\usepackage{hyperref}

\newcommand{\Ca}{\mathrm{card}}
\newcommand{\e}{\varepsilon}
\newcommand{\di}{\mathrm{d}}

\newcommand{\alp}{\omega+2}

\newtheorem{remarks}{Remarks}[section]

\newtheorem{theorem}{Theorem}[section]

\newtheorem{corollary}{Corollary}[section]
\newtheorem{lemma}{Lemma}[section]

\begin{document}
\title{Equivariant Yamabe problem and Hebey--Vaugon conjecture}
\author{Farid Madani\\
}
\date{}

\address{Institut de Math\'ematiques de Jussieu, Universit\'e Pierre et Marie Curie\\
\'Equipe: d'Analyse Complexe et G\'eom\'etrie \\
175, rue Chevaleret\\
75013 Paris, France.}
\curraddr{}
\email{\href{mailto:madani@math.jussieu.fr}{madani@math.jussieu.fr}}



\maketitle

\begin{abstract}
In their study of the Yamabe problem in the presence of isometry group, E.~Hebey and M.~Vaugon announced a conjecture. This conjecture generalizes  T.~Aubin's conjecture, which has already been proven and is sufficient to solve the Yamabe problem.  In this paper,  we generalize  Aubin's theorem and we prove the Hebey--Vaugon conjecture in some new cases. 
\end{abstract}

\section{Introduction} 
Let $(M,g)$ be a compact Riemannian manifold of dimension $n\geq 3$. Denote by $I(M,g)$, $C(M,g)$ and $R_g$ the isometry group, the conformal transformations group and the scalar curvature, respectively. Let $G$ be a subgroup of the isometry group $I(M,g)$. E.~Hebey and M.~Vaugon\cite{HV} considered the following problem: 

\medskip

\paragraph{\textsc{Hebey--Vaugon problem}}\emph{Is there some $G-$invariant metric $g_0$ which minimizes the functional
$$J(g')=\frac{\int_MR_{g'}\di v(g')}{(\int_M\di v(g'))^{\frac{n-2}{n}}}$$
where $g'$ belongs to the $G-$invariant conformal class of metrics $g$ defined by:}
 $$[g]^G:=\{\tilde g=e^fg/ f\in C^\infty(M),\; \sigma^*\tilde g=\tilde g\quad \forall \sigma\in G\}$$
The positive answer would have two consequences. The first is that there exists an $I(M,g)-$invariant metric $g_0$ conformal to $g$ such that the scalar curvature $R_{g_0}$ is constant. The second is that the A.~Lichnerowicz's conjecture \cite{Lic}, stated below, is true. By the works of J.~Lelong-Ferrand\cite{Lel} and M.~Obata\cite{Oba}, we know that if $(M,g)$ is not conformal to $(S_n,g_{can})$ (the unit sphere endowed with its standard metric $g_{can}$), then $C(M,g)$ is compact and there exists a conformal metric $g'$ to $g$ such that $I(M,g')=C(M,g)$. This implies that the first consequence is equivalent to the

\medskip

\paragraph{\textsc{A.~Lichnerowicz conjecture}}\emph{ For every compact Riemannian manifold $(M,g)$ which is not conformal to the unit sphere $S_n$ endowed with its standard metric, there exists a metric $\tilde g$ conformal to $g$ for which $I(M,\tilde g)=C(M,g)$, and the scalar curvature $R_{\tilde g}$ is constant.}\\

To such metrics correspond  functions which are necessarily solutions of the Yamabe equation. In other words, if $\tilde g=\psi^{\frac{4}{n-2}}g$, $\psi$ is a $G-$invariant smooth positive function then $\psi$ satisfies 
$$\frac{4(n-1)}{n-2}\Delta_g \psi+ R_g\psi= R_{\tilde g} \psi^{\frac{n+2}{n-2}}.$$
The classical Yamabe problem, which consists to find a conformal metric with constant scalar curvature on a compact Riemannian manifold, is the particular case of the problem above when $G=\{\mathrm{id}\}$. Denote by $O_G(P)$ the orbit of $P\in M$ under $G$, $W_g$ the Weyl tensor associated to the manifold $(M,g)$ and $\omega_n$ the volume of the unit sphere $S_n$. We define the integer $\omega(P)$ at the point $P$ as 
$$\omega(P)=\inf \{i\in \mathbb N/\|\nabla^i W_g(P)\|\neq 0\}\; (\omega(P)=+\infty\text{ if }\forall i\in \mathbb N,\;\|\nabla^i W_g(P)\|=0)$$

\medskip

\paragraph{\textsc{Hebey--Vaugon conjecture}}\emph{Let $(M,g)$ be a compact Riemannian manifold of dimension $n\geq 3$ and $G$ be a subgroup of $I(M,g)$. If $(M,g)$ is not conformal to $(S_n, g_{can})$ or if the action of $G$ has no fixed point, then the following inequality holds }
\begin{equation}\label{HVI}
\inf_{g'\in [g]^G} J(g')<n(n-1)\omega_n^{2/n}(\inf_{Q\in M}\Ca O_G(Q))^{2/n}
\end{equation}
\medskip

\begin{remarks}\label{remrem}
\begin{enumerate} 
\item This conjecture is the generalization of the former T.~Aubin's conjecture \cite{Aub} for the Yamabe problem corresponding to $G=\{\mathrm{id}\}$, where the constant in the right side of the inequality is equal to $\inf_{g'\in [g_{can}]} J(g')$ for $S_n$. In this case, the conjecture is completely proved.
\item The inequality is obvious if $\inf_{g'\in [g]^G} J(g')$ is nonpositive, it is the case when there exists a Yamabe metric with nonpositive scalar curvature. 
\item If for any $Q\in M$,  $\mathrm{card} O_G(Q)=+\infty$ then this conjecture is also obvious. 
\end{enumerate}
\end{remarks}
The only results known about this conjecture are given in the following theorem:
\begin{theorem}[E.~Hebey and M.~Vaugon]\label{HV theorem}
Let $(M,g)$ be a smooth compact Riemannian manifold of dimension $n\geq 3$  and $G$ be a subgroup of $I(M,g)$.
We always have :  
$$\inf_{g'\in [g]^G} J(g')\leq n(n-1)\omega_n^{2/n}(\inf_{Q\in M}\Ca O_G(Q))^{2/n}$$
 and inequality \eqref{HVI} holds if one of the following items is satisfied. 
 \begin{enumerate}
 \item The action of $G$  on $M$ is free
 \item $3\leq \dim M\leq 11$
 \item\label{item 3}There exists a point $P$ with minimal orbit (finite) under $G$ such that $\omega(P)>(n-6)/2$ or $\omega(P)\in \{0,1,2\}$.
 \end{enumerate}
\end{theorem}
The case $\omega=3$ was studied by A. Rauzy (private communications).\\

In this prove we prove the following results:

\paragraph{\textbf{Main theorem}}
\emph{The Hebey--Vaugon conjecture holds  if there exists a point $P\in M$ with minimal orbit (finite) for which $\omega(P)\leq15$ or if the degree of the leading part of $R_g$ is greater or equal to $\omega(P)+1$, in the neighborhood of this point $P$.}\\

\begin{corollary}\label{cor prin}
Hebey--Vaugon conjecture  holds  for every smooth compact Riemannian manifold $(M,g)$ of dimension $n\in [ 3, 37 ]$.
\end{corollary} 

To prove the main theorem, we need to construct a $G-$invariant test function $\phi$ such that  
$$I_g(\phi)< n(n-1)\omega_n^{2/n}(\inf_{Q\in M}\Ca O_G(Q))^{2/n}$$
Thus, all  the difficulties are in the construction of a such function. For some cases, we can use the test functions constructed by T.~Aubin \cite{Aub} and R.~Schoen~\cite{Schoen} in the case of Yamabe problem. They have been already proven by E.~Hebey and M.~Vaugon \cite{HV}.  But the item \ref{item 3}, presented in  Theorem \ref{HV theorem}, uses test functions  different than  T.~Aubin and R.~Schoen ones.\\ 
 
We multiply  T.~Aubin's test function $u_{\e,P}$ by a function as follows:
\begin{equation}\label{phiepsilon}
\varphi_{\e}(Q)=(1- r^{\alp} f(\xi))u_{\e,P}(Q)
\end{equation}
\begin{equation}\label{uepsilon}
u_{\e,P}(Q)=\begin{cases}\biggl(\displaystyle\frac{\varepsilon}{r^2+\varepsilon^2}\biggr)^{\frac{n-2}{2}}-\biggl(\frac{\varepsilon}{\delta^2+\varepsilon^2}\biggr)^{\frac{n-2}{2}} &\mbox{ if }Q\in B_{P}(\delta)\\
0 &\mbox{ if }Q\in M-B_{P}(\delta)
\end{cases}
\end{equation}
for all $Q\in M$, where $r=d(Q,P)$ is the distance between $P$ and $Q$. $(r,\xi^j)$ is a geodesic coordinates system  in the neighborhood of $P$ and 
$B_{P}(\delta)$ is the geodesic ball of center $P$ with radius $\delta$ fixed sufficiently small. $f$ is a function  depending only on $\xi$, chosen such that $\int_{S_{n-1}}f d\sigma=0$. Without loss of generality, we  suppose that in the coordinates system $(r,\xi^j)$ we have $\det g=1+o(r^m)$ for $m\gg 1$. In fact, E.~Hebey and M.~Vaugon proved that there exists $\tilde g\in [g]^G$ for which $\det \tilde g=1+o(r^m)$ and $\inf_{g'\in [g]^G} J(g')$ does not depend on the conformal $G-$invariant metric.


\section{Computation of $\int_M R_g\varphi_\e^2dv$ }\label{computations}

\noindent Let be
\begin{equation*}
I_a^b(\e)=\int_0^{\delta/\e}\frac{t^b}{(1+t^2)^a}dt\text{ and }I_a^b=\lim_{\e\to 0}I_a^b(\e)
\end{equation*}
then $I_a^{2a-1}(\e)=\log\e^{-1}+O(1)$. If $2a-b>1$ then $I_a^b(\e)=I_a^b+O(\e^{2a-b-1})$  and by integration by parts, we establish the following relationships : 
\begin{equation}\label{rela}
 I_a^b=\frac{b-1}{2a-b-1}I_a^{b-2}=\frac{b-1}{2a-2}I_{a-1}^{b-2}=\frac{2a-b-3}{2a-2}I_{a-1}^{b},\quad \frac{4(n-2)I_n^{n+1}}{(I^{n-2}_n)^{(n-2)/n}}=n
\end{equation}
Using the inequality $(a-b)^\beta\geq a^\beta-\beta a^{\beta-1}b$ for $0<b<a$, we have for $\beta\geq 2$, $0\leq \alpha< (n-2)(\beta-1)-n$ 
\begin{equation}\label{hhh}
 \int_Mr^\alpha u_{\e,P}^\beta \di v=\omega_{n-1}I_{(n-2)\beta/2}^{\alpha+n-1}\e^{\alpha+n-\beta(n-2)/2}+O(\e^{n-2})
\end{equation}
This integral appears frequently in the following computations, and it allows us to neglect the constant term in the expression of $u_\e$, when we choose $\delta$ sufficiently small and $\e$ smaller than $\delta$. \\
\noindent Denote by $I_g$ the Yamabe functional defined for all 
$\psi\in H^1(M)$ by 
\begin{equation}\label{yamm}
I_g(\psi)=\biggr(\int_M|\nabla_g\psi|^2\di v+\frac{(n-2)}{4(n-1)}\int_MR_g\psi^2\di v\biggl)\|\psi\|_N^{-2} 
\end{equation}
where $N=2n/(n-2)$ and $\nabla_g$ is the gradient of the metric $g$.\\

\noindent The second integral of the functional $I_g$ with the scalar curvature term needs a special consideration. Let $\mu(P)$ be an integer defined as follows : $|\nabla_\beta R_g(P)|=0$ for all $|\beta|<\mu(P)$ and there exists $\gamma\in \mathbb N^{\mu(P)}$ such that $|\nabla_\gamma R_g(P)|\neq 0$ 
then 
$$R_g(Q)=\bar R +O(r^{\mu(P)+1})$$ 
where  $\bar R=r^{\mu(P)}\sum_{|\beta|= \mu}\nabla_\beta R_g(P)\xi^\beta$ is a homogeneous polynomial of degree $\mu(P)$, the $\beta$ are multi-indices. \\

For simplicity, we drop the letter $P$ in $\omega(P)$ and $\mu(P)$.\\

By E.~Hebey and M.~Vaugon~\cite{HV} results:
\begin{lemma}\label{LLLL}
 $\mu\geq \omega$, $g_{ij}=\delta_{ij}+O(r^{\omega+2})$ and $\bar\int_{S(r)}R_g=O(r^{2\omega+2})$ which implies that  $\int_{S(r)}\bar Rd\sigma=~0$ when $\mu<2\omega+2$
\end{lemma}
$\bar\int$ denotes the average. Then 

\begin{equation}\label{scal}
\begin{split}
\int_MR_g\varphi_\e^2\di v & =\int_MR_gu_{\e,P}^2\di v-2\int_Mfu_{\e,P}^2R_gr^{\omega+2} \di v+\int_Mf^2u_{\e,P}^2R_gr^{2\omega+4} \di v\\
& =\e^{2\omega+4}\omega_{n-1}\bar\int_{S(r)}r^{-2\omega-2}R_g\di \sigma I_{n-2}^{n+2\omega+1}(\e)-\\ 
& 2\e^{\omega+\mu+4}I_{n-2}^{\omega+\mu+n+1}(\e)\omega_{n-1}\bar\int_{S(r)}\hspace{-0.3cm}r^{-\mu}f(\xi)\bar R \di\sigma(\xi)+O(\e^{n-2})\\
\end{split}
\end{equation}

Moreover T.~Aubin \cite{Aub3} proved that:


\begin{theorem}\label{aaa}
 If $\mu\geq\omega+1$ then there exists $C(n,\omega)>0$ such that 
$$\bar\int_{S_{n-1}(r)}R\di\sigma=C(n,\omega)(-\Delta_g)^{\omega+1} R(P)r^{2\omega+2}+o(r^{2\omega+2})$$
$(-\Delta_g)^{\omega+1} R(P)$ is negative. Then $I_g(u_{\e,P})<\frac{n(n-2)}{4}\omega_{n-1}^{2/n}$.\\
\end{theorem}

From now until the end of this section, we make the assumption that $\mu=\omega$. Now, we recall some results obtained by T.~Aubin in his papers \cite{Aub5, Aub4}:\\

$\bar R$ is homogeneous polynomial of degree  $\omega$ then $\Delta_{\mathcal{E}}\bar R$ is homogeneous of degree $\omega-2$ and
$$\Delta_{\mathcal{E}}\bar R=r^{-2}(\Delta_s\bar R-\omega(n+\omega-2) \bar R)$$
where $\Delta_{\mathcal{E}}$ is the Euclidean Laplacian and  $\Delta_s$ is the Laplacian on the sphere $S_{n-1}$. $\Delta_{\mathcal{E}}^{k-1}\bar R$ is homogeneous of degree $\omega-2k+2$ and
$$\Delta^k_{\mathcal E}\bar R=r^{-2}(\Delta_s-\nu_k\mathrm{id})\Delta^{k-1}_{\mathcal E}\bar R= r^{-2k}\prod_{p=1}^k(\Delta_S-\nu_p\mathrm{id}) \bar R$$
with 
\begin{equation}\label{lam}
 \nu_k=(\omega-2k+2)(n+\omega-2k)
\end{equation}
The sequence of integers $(\nu_k)_{\{1\leq k\leq [\omega/2]\}}$ is  decreasing. It will play the role of the eigenvalues of the Laplacian on the sphere $S_{n-1}$. It is known that the eigenvalues of the geometric Laplacian are non-negative and increasing. Our $\nu_k$ are in the opposite order. \\
We know by T.~Aubin's paper \cite{Aub3} that $\Delta^{[\omega/2]}_{\mathcal E}\bar R=0$ and $\int_{S(r)}\bar R\di\sigma=0$, then 
$$q=\min\{k\in\mathbb N/ \Delta_{\mathcal E}^{k}\bar R= 0\}$$ 
is well defined and $r^{-\omega}\bar R\in \bigoplus_{k=1}^{q}E_k $, with $E_k$ the eigenspace associated to the positive eigenvalues $\nu_k$ of the Laplacian $\Delta_s$ on the sphere $S_{n-1}$. If $j\neq k$, then $E_k$ is orthogonal to $E_j$, for the standard scalar product in $H_1^2(S_{n-1})$. Moreover, since $\int \bar R d\sigma=0$ there exist $\varphi_k\in E_k$ (eigenfunctions of $\Delta_s$) such that
\begin{equation}\label{dddd}
\bar R=r^{\omega}\Delta_s\sum_{k=1}^q\varphi_k=r^{\omega}\sum_{k=1}^q\nu_k\varphi_k
\end{equation}

According to Lemma \ref{LLLL}, we can split the metric $g$ in the following way: 
\begin{equation}
 g=\mathcal E+h
\end{equation}
where $\mathcal E$ is the Euclidean metric and $h$ is a symmetric 2-tensor defined in our geodesic coordinates system  by
\begin{equation}
h_{ij}=r^{\omega+2}\bar g_{ij}+r^{2(\omega+2)}\hat{g}_{ij}+\tilde h_{ij}\text{ and }h_{ir}=h_{rr}=0
\end{equation}
where $\bar g$, $\hat g$ and $\tilde h$ are symmetric 2-tensors defined on the sphere $S_{n-1}$. We denote by $s$ the standard metric on the sphere, $\nabla$, $\Delta$ are the associated gradient and Laplacian on $S_{n-1}$. By straightforward computations, Aubin \cite{Aub5} proved that:

\begin{lemma}\label{zzz}
$$\bar R=\nabla^{ij}\bar g_{ij}r^\omega\text{ and }$$
$$\bar\int_{S_{n-1}(r)}\hspace{-0.7cm}R\di\sigma =[B/2-C/4-(1+\omega/2)^2Q]r^{2(\omega+1)}+o(r^{2(\omega+1)})$$
where $B=\bar\int_{S_{n-1}}\hspace{-0.5cm}\nabla^i\bar g^{jk}\nabla_j\bar g_{ik}\di\sigma$, $C=\bar\int_{S_{n-1}}\hspace{-0.5cm}\nabla^i\bar g^{jk}\nabla_i\bar g_{jk}\di\sigma$ and $Q=
\bar\int_{S_{n-1}}\bar g_{ij}\bar g^{ij}\di\sigma$
\end{lemma}
For further details refer to \cite{Mad3}.\\
The integrals $Q$, $B$ and $C$ are given in terms of the tensor $\bar g$. Our goal is to compute them using the eigenfunctions $\varphi_k$ above. Let us define 
 
 $$b_{ij}=\sum_{k=1}^q\frac{1}{(n-2)(\nu_k+1-n)}[(n-1)\nabla_{ij}\varphi_k+\nu_k\varphi_ks_{ij}]$$

and $a_{ij}$ such that $\bar g_{ij}=a_{ij}+b_{ij}$ then, according to \eqref{dddd}, we check that 
\begin{equation}
\bar R=\bar R_b=\nabla^{ij}b_{ij}r^{\omega}\quad\text{and } \bar R_a=\nabla^{ij}a_{ij}r^{\omega}=0
\end{equation}

\noindent If $\bar g_{ij}=a_{ij}$ then $\bar R=\bar R_a=0$ and $\mu\geq \omega+1$. By Theorem \ref{aaa}
$$\bar\int_{S_{n-1}(r)}R\di\sigma=\bar\int_{S_{n-1}(r)}R_a\di\sigma<0$$ 
If $\bar g_{ij}=b_{ij}$ then
$$\bar\int_{S_{n-1}(r)}\hspace{-0.5cm}R\di\sigma=\bar\int_{S_{n-1}(r)}\hspace{-0.5cm}R_b\di\sigma=[B_b/2-C_b/4-(1+\omega/2)^2Q_b]r^{2(\omega+1)}+o(r^{2(\omega+1)})$$
where $B_b$, $C_b$ and $Q_b$ are the same integrals defined in Lemma \ref{zzz} when the considered tensor $\bar g_{ij}=b_{ij}$. We  compute them in terms of $\varphi_k$
\begin{gather*}
Q_b=\bar\int_{S_{n-1}}\bar b_{ij}\bar b^{ij}\di\sigma=\frac{n-1}{n-2}\sum_{k=1}^{q}\frac{\nu_k}{\nu_k-n+1}\bar\int_{S_{n-1}}\varphi_k^2\di\sigma\\
B_b=-(n-1)Q_b+\sum_{k=1}^{q}\nu_k\bar\int_{S_{n-1}}\hspace{-0.6cm}\varphi_k^2\di\sigma\\ C_b=-(n-1)Q_b+\frac{n-1}{n-2}\sum_{k=1}^{q}\nu_k\bar\int_{S_{n-1}}\hspace{-0.6cm}\varphi_k^2\di\sigma
\end{gather*}
To find these expressions, we used several times the identity $\nabla^ib_{ij}=-\sum_{k=1}^q\nabla_j\varphi_k$ and Stokes formula (more details are given in \cite{Aub5,Aub4} and \cite{Mad3}). In the general case, we deduce that
\begin{lemma}\label{mmm}
If  $\mu=\omega$ and $\bar g_{ij}=a_{ij}+b_{ij}$, where $b_{ij}$ is defined above, 
\begin{equation}
\bar\int_{S_{n-1}(r)}\hspace{-1cm}R\di\sigma=\bar\int_{S_{n-1}(r)}\hspace{-1cm}R_a+R_b\di\sigma\leq [B_b/2-C_b/4-(1+\omega/2)^2Q_b]r^{2(\omega+1)}+o(r^{2(\omega+1)})
\end{equation} 
and
\begin{equation}
B_b/2-C_b/4-(1+\omega/2)^2Q_b=\sum_{k=1}^q u_k\bar\int_{S_{n-1}}\varphi_k^2\di\sigma
\end{equation} 
with 
\begin{equation}
 u_k=\biggl(\frac{n-3}{4(n-2)}-\frac{(n-1)^2+(n-1)(\omega+2)^2}{4(n-2)(\nu_k-n+1)}\biggr)\nu_k
\end{equation}

\end{lemma}
$u_k$ is obtained using the expressions of $Q_b$, $B_b$ and $C_b$ above.
\vspace{0.8cm}


\section{Generalization of T.~Aubin's theorem}

\begin{theorem}\label{theo}
 If there exists $P\in M$ such that  $\omega(P)\leq (n-6)/2$ then there exists $f\in C^\infty(S_{n-1})$ with vanishing mean integral  such that
$$I_g(\varphi_\e)< \frac{n(n-2)}{4}\omega_{n-1}^{2/n}$$
\end{theorem}

The case $\omega= 0$ of the this theorem has already been proven by T.~Aubin~\cite{Aub}. He also proved the theorem when $\mu\geq \omega+1$ (see Theorem \ref{aaa}).\\

From now until the end of this paper, we drop the letter $P$ in $\omega(P)$ and $\mu(P)$.
\begin{proof}
If $\mu\geq\omega+1$ then the inequality holds by Theorem \ref{aaa}.  So we suppose that $\mu=\omega$ until the end of the proof.  We start by computing the first integral of the Yamabe functional \eqref{yamm} with $\psi=\varphi_\e$. Using formula $|\nabla_g\varphi_\e|^2=(\partial_r\varphi_\e)^2+r^{-2}|\nabla_s\varphi_\e|^2$, we obtain:
\begin{multline*}
\int_M|\nabla_g\varphi_\e|^2\di v=\int_M|\nabla_g u_{\e,P}|^2\di v+\int_0^\delta[\partial_r(r^{(\alp)} u_{\e,P})]^2r^{n-1}\di r\int_{S_{n-1}}
\hspace{-0.5cm}f^2 \di\sigma +\\
\int_0^\delta u^2_{\e,P} r^{n+2\omega+1}\di r\int_{S_{n-1}}\hspace{-0.5cm} |\nabla f|^2 \di \sigma 
\end{multline*}
The substitution $t=r/\e$ gives
\begin{multline}\label{grad}
\int_M|\nabla_g\varphi_\e|^2\di v=(n-2)^2\omega_{n-1}I_n^{n+1}(\e)+\e^{2\omega+4}\biggl\{\int_{S_{n-1}}|\nabla f|^2 \di\sigma I_{n-2}^{2\omega+n+1}(\e)+\\
\int_{S_{n-1}}\hspace{-0.5cm}f^2\di\sigma [(\omega-n+4)^2I_n^{2\omega+n+5}(\e)+2(\alp)(\omega-n+4)I_n^{2\omega+n+3}(\e)+(\alp)^2I_n^{2\omega+n+1}(\e)]\biggr\}
\end{multline}
For $\|\varphi_\e\|_N^{-2}$, we need to  compute the Taylor expansion of :
$$\varphi_\e^N(Q)=[1-Nr^{\omega+2}f(\xi)+\frac{N(N-1)}{2}r^{2\omega+4}f^2(\xi)+o(r^{2\omega+4})]u_{\e,P}^N$$
Using the fact that $\int_{S_{n-1}} f\di\sigma(\xi)=0$ and  formula \eqref{hhh},  we conclude that 
\begin{equation*}
\begin{split}
\|\varphi_\e\|_N^N &=\int_0^\delta\int_{S_{n-1}}\hspace{-0.5cm}[1+\frac{N(N-1)}{2} r^{2(\alp)} f^2(\xi)+o(r^{2\omega+4})]r^{n-1}u^N_{\e,P} \di r\di\sigma(\xi)\\
& = \omega_{n-1}I^{n-1}_n+\frac{N(N-1)}{2}\e^{2(\alp)}\int_{S_{n-1}}\hspace{-0.5cm}f^2 \di\sigma I_n^{2\omega+n+3}+ o(\e^{2\omega+4})
\end{split} 
\end{equation*}
then
\begin{multline}\label{norm}
\|\varphi_\e\|_N^{-2}=(\omega_{n-1}I^{n-1}_n)^{-2/N}\bigl\{1\\
-(N-1)\e^{2(\alp)}\int_{S_{n-1}}\hspace{-0.5cm}f^2 \di\sigma I_n^{2\omega+n+3}/(\omega_{n-1}I^{n-1}_n)\bigr\}+o(\e^{2\omega+4})
\end{multline}

By  Eqs \eqref{grad}, \eqref{norm}, \eqref{scal} and the relationship \eqref{rela}, if $n>2\omega+6$ then :

\begin{multline*}
I_g(\varphi_\e)= \frac{n(n-2)}{4}\omega_{n-1}^{2/n}+(\omega_{n-1}I^{n-1}_n)^{-2/N}I_{n-2}^{n+2\omega+1}\e^{2\omega+4}\times \\
\biggl\{\frac{(n-2)\omega_{n-1}}{4(n-1)}\bar\int_{S(r)}r^{-2\omega-2}R_g\di \sigma -\frac{n-2}{2(n-1)}\int_{S_{n-1}}\hspace{-0.3cm}f(\xi)\bar R \di\sigma+\int_{S_{n-1}}\hspace{-0.5cm}|\nabla f|^2 \di\sigma +\\
-\frac{n(n-2)^2-(\omega+2)^2(n^2+n+2)}{(n-1)(n-2)}\int_{S_{n-1}}\hspace{-0.5cm}f^2\di\sigma\biggr\}+
o(\e^{2\omega+4})
\end{multline*}

If $n=2\omega+6$ then

\begin{multline*}
I_g(\varphi_\e)= \frac{n(n-2)}{4}\omega_{n-1}^{2/n}+(\omega_{n-1}I^{n-1}_n)^{-2/N}\e^{2\omega+4}\log\e^{-1}\times\\
\biggl\{\frac{(n-2)\omega_{n-1}}{4(n-1)}\bar\int_{S(r)}r^{-2\omega-2}R_g\di \sigma
-\frac{n-2}{2(n-1)}\int_{S_{n-1}}\hspace{-0.3cm}f(\xi)\bar R \di\sigma+\\
\int_{S_{n-1}}|\nabla f|^2 \di\sigma +(\omega+2)^2\int_{S_{n-1}}\hspace{-0.5cm}f^2\di\sigma\biggr\}+O(\e^{2\omega+4})
\end{multline*}

For further details refer to  \cite{Mad3}.\\

\noindent Let $I_S$ be the functional defined for a function $f$ on the sphere $S_{n-1}$, with zero mean integral , by
\begin{multline*}
I_S(f)=\bar\int_{S_{n-1}} \hspace{-0.3cm}4(n-1)(n-2)|\nabla f|^2-[4n(n-2)^2-4(\omega+2)^2(n^2+n+2)]f^2+\\
-2(n-2)^2f\bar R\di\sigma
\end{multline*}
This implies that if $n>2\omega+6$ 
\begin{multline}\label{eqeq1}
I_g(\varphi_\e)=\frac{n(n-2)}{4}\omega_{n-1}^{2/n}+\frac{\omega_{n-1}^{2/n}I_{n-2}^{n+2\omega+1}\e^{2\omega+4}} {4(n-1)(n-2)(I^{n-1}_n)^{2/N}}\times\\
\{(n-2)^2\bar\int_{S(r)}r^{-2\omega-2}R_g\di\sigma+I_S(f)\}+o(\e^{2\omega+4}) 
\end{multline}
and if $n=2\omega+6$ 
\begin{multline}\label{eqeq2}
I_g(\varphi_\e)=\frac{n(n-2)}{4}\omega_{n-1}^{2/n}+ 
\frac{\omega_{n-1}^{2/n}I_{n-2}^{n+2\omega+1}\e^{2\omega+4}\log\e^{-1}} {4(n-1)(n-2)(I^{n-1}_n)^{2/N}}\times\\
\{(n-2)^2\bar\int_{S(r)}r^{-2\omega-2}R_g\di\sigma+I_S(f)\}+O(\e^{2\omega+4}) 
\end{multline}

Notice that if $k\neq j$ then $I_S(\varphi_k+\varphi_j)=I_S(\varphi_k)+I_S(\varphi_j)$. Indeed, $\varphi_k$ and $\varphi_j$ are orthogonal for the standard scalar product in $H_1^2(S_{n-1})$.  
\begin{equation*}
\begin{split} 
I_S(c_k\nu_k\varphi_k) &=\bigl\{d_kc_k^2-2(n-2)^2c_k \bigr\}\nu_k^2\bar\int_{S_{n-1}}\varphi_k^2\di\sigma\\
 & =-\frac{(n-2)^4}{d_k}\nu_k^2\bar\int_{S_{n-1}}\varphi_k^2\di\sigma
\end{split}
\end{equation*} 
 
where $$d_k=4[(n-1)(n-2)\nu_k-n(n-2)^2+(\omega+2)^2(n^2+n+2)]\text{ and }c_k=\frac{(n-2)^2}{d_k}$$ 
Using \eqref{lam}, we can check easily  that $d_k$ is positive for any $1\leq k\leq [\omega/2]$.
Now, let us consider $f=\sum_{1}^qc_k\nu_k\varphi_k$. Then 
$$I_S(f)=-\sum_{1}^q\frac{(n-2)^4}{d_k}\nu_k^2\bar\int_{S_{n-1}}\varphi_k^2\di\sigma$$ 
and by  Lemma \ref{mmm}

$$(n-2)^2\bar\int_{S(r)}\hspace{-0.5cm}r^{-2\omega-2}R_g\di\sigma+I_S(f)\leq\sum_{1}^q(u_k(n-2)^2-\frac{(n-2)^4}{d_k}\nu_k^2)\bar\int_{S_{n-1}}\hspace{-0.5cm}\varphi_k^2\di\sigma+o(1)$$
The following lemma implies that 
$I_g(\varphi_\e)<\frac{n(n-2)}{4}\omega_{n-1}^{2/n}$
\end{proof}

\begin{lemma}\label{lemme poly}
For any $k\leq q\leq [\omega/2]$ the following inequality holds
$$ u_k-\frac{(n-2)^2}{d_k}\nu_k^2< 0$$
\end{lemma}

\begin{proof}
Recall the expression of $\nu_k$ given in \eqref{lam}. The sequence $(U_k)$ defined by 
$$U_k:=(\nu_k-n+1)d_k\{(n-2)\frac{u_k}{\nu_k}-\frac{(n-2)^3}{d_k}\nu_k\}$$ 
is polynomial decreasing in $\nu_k$ when $\nu_k\geq 0$. In fact, $U_k=P(\nu_k)$ with $P$ the decreasing polynomial in $\mathbb R_+$, defined by
\begin{multline*}
 P(x)=[(n-1)(n-2)x-n(n-2)^2+(\omega+2)^2(n^2+n+2)]\times\\
[(n-3)(x-n+1)-(n-1)^2-(n-1)(\omega+2)^2]-(n-2)^3(x^2-(n-1)x)
\end{multline*}
The derivative of $P$ is
\begin{equation*}
 P'(x)=-2(n-2)x-2n(n-2)^3+2(n^2-3n-2)(\omega+2)^2
\end{equation*}
By assumption $\omega+2\leq (n-2)/2$ then $P$ is decreasing in $\mathbb R_+$. Hence 
$$U_k=P(\nu_k)\leq P(\nu_{\omega/2})=U_{\omega/2}$$ 
for all $k\leq \omega/2$. It easy to check that $u_{\omega/2}$ is negative so $U_k\leq U_{\omega/2}<0$.
\end{proof}

\section{Proof of the main theorem}

By Remarks \ref{remrem}, we consider only the positive case (i.e., $\inf_{g'\in [g]^G} J(g')>0$) and the case when there exists $P\in M $ such that  
$$O_G(P)=\{P_i\}_{1\leq i\leq m},\;\; m=\Ca O_G(P)=\inf_{Q\in M}\Ca O_G(Q),\;\omega\leq\frac{n-6}{2}\text{ and }P_1=P$$
Let $\tilde\varphi_{\e, i}$ be a function defined as follows:

\begin{equation}\label{varvar}
\tilde\varphi_{\e,i}(Q)=(1- r_i^{\alp} f_i(\xi))u_{\e,P_i}(Q)
\end{equation}
where $r_i=d(Q,P_i)$, the function $u_{\e,P_i}$ is defined as in \eqref{uepsilon} and $f_i$ is defined by:
\begin{equation}\label{fitil}
 f_i(Q)=c r_i^{-\omega}\nabla_g^\omega R_{(P_i)}(\exp_{P_i}^{-1}Q,\cdots,\exp_{P_i}^{-1}Q)
\end{equation} 
$\exp_{P_i}$ is the exponential map. In a geodesic coordinates system $\{r,\xi^j\}$ with origin $P$, induced by the exponential map
\begin{equation*}
f_1=c r^{-\omega}\bar R=c\sum_{k=1}^q\nu_k\varphi_k
\end{equation*}

where $\bar R$, $\varphi_k$ and $\nu_k$ are defined in  Section \ref{computations}. Thus the functions $f_i$ are defined on the sphere $S_{n-1}$. The choice of the constant $c$ is important. 

\begin{lemma}\label{lemme infg}
Suppose that  $\omega\leq (n-6)/2$. If $\omega\in [3,15]$ or if $\mathrm{deg}\bar R\geq\omega+1$ then there exists $c\in \mathbb R$ such that the corresponding functions $\tilde\varphi_{\e,i}$ satisfy :
\begin{equation}\label{inegde}
I_g(\tilde\varphi_{\e,i})<\frac{1}{4} n(n-2)\omega_{n}^{2/n}
\end{equation} 
\end{lemma}

\begin{remarks}
\begin{enumerate} 
\item We proved inequality of this lemma for any   $\omega\leq (n-6)/2$, using test function  $\varphi_\e$ (see Theorem \ref{theo}).  We notice that the  difference between  $\varphi_\e$ and $\tilde\varphi_{\e,i}$ is on the construction of the corresponding functions $f$ and $f_i$ respectively. From $\tilde\varphi_{\e,i}$ we define a $G-$invariant function (see proof of the main theorem below), this property is not possible with the function  $\varphi_\e$.
\item For $\omega=16$ and $n$  sufficiently big, we can check that for any $c\in\mathbb R$, inequality \eqref{inegde} is false. 
\end{enumerate}
\end{remarks}
\begin{proof}
1. If $\mathrm{deg}\bar R\geq \omega+1 $, then by Theorem  \ref{aaa}
$$I_g(u_{\e,P_i})<\frac{n(n-2)}{4}\omega_{n}^{2/n}$$
It is sufficient to take  $c=0$, hence $\tilde\varphi_{\e,i}=u_{\e,P_i}$.\\

2. If $\mathrm{deg}\bar R= \omega $. Using estimates given in the proof of Theorem \ref{theo} (see \eqref{eqeq1}, \eqref{eqeq2}), it is sufficient to show that there exists  $c\in\mathbb R$ such that
\begin{equation}\label{gfgf}
I_S( f_1)+(n-2)^2\bar\int_{S(r)}r^{-2\omega-2}R_g\di\sigma_r<0
\end{equation}
We keep the notations used in the proof of Theorem \ref{theo}. Thus
\begin{equation*}
I_S(f_1)=\sum_{k=1}^q I_S(c\nu_k\varphi_k)=\bigl\{d_kc^2-2(n-2)^2c \bigr\}\nu_k^2\bar\int_{S_{n-1}}\varphi_k^2\di\sigma
\end{equation*}
\begin{equation*}
\text{and }\bar\int_{S(r)}r^{-2\omega-2}R_g\di\sigma_r=\sum_{k=1}^q u_k\bar\int_{S_{n-1}}\varphi_k^2\di\sigma
\end{equation*}
To prove inequality \eqref{gfgf}, it is sufficient to prove that
\begin{equation}\label{ineggg}
\forall k\leq q \quad \frac{d_k}{2(n-2)}c^2-(n-2)c+(n-2)\frac{u_k}{2\nu_k^2}<0
\end{equation} 
The left side of the inequality above is a second degree polynomial with variable $c$, his discriminant is:
\begin{equation}\label{deltak}
\Delta_k=(n-2)^2-\frac{d_ku_k}{\nu_k^2}
\end{equation}
Using Lemma \ref{lemme poly}, we deduce that for any $k\leq q$, $\Delta_k>0$. Hence, the polynomial above admits two different roots  denoted $x_k<y_k$ and given by
\begin{equation*}
x_k=\frac{(n-2)^2-(n-2)\sqrt{\Delta_k}}{d_k},\qquad y_k=\frac{(n-2)^2+(n-2)\sqrt{\Delta_k}}{d_k}
\end{equation*}
Inequality \eqref{ineggg} holds if and only if
\begin{equation}\label{propo}
\bigcap_{k=1}^q(x_k,y_k)\neq \varnothing
\end{equation}
The sequence $(d_k)_{k\leq[\omega/2]}$ decreases. It is easy to check that
\begin{equation}\label{ine solut}
 \forall k< j\leq [\frac{\omega}{2}]\qquad x_k< y_j
\end{equation}
Hence intersection \eqref{propo} is not empty if 
\begin{equation}\label{ine solut2}
 \forall k< j\leq [\frac{\omega}{2}]\qquad x_j< y_k
\end{equation}
We also check that  if  $\omega$ is even,   $u_{\omega/2}<0$, which implies $x_{\omega/2}<0$.
\begin{itemize}
\item[$i.$] If $\omega=3$ then $q=1$, intersection above is not empty. It is sufficient to take $c=(x_1+y_2)/2$. 
\item[$ii.$] If $\omega=4$ then $k\in\{1,2\}$, $x_2<0$ (because $u_2<0$) and $0<x_1<y_2$. Hence intersection $]x_1,y_1[\cap]x_2,y_2[$ is not empty.
\item[$iii.$]If $5\leq \omega\leq 15$, it is sufficient to prove \eqref{ine solut2} which is equivalent to prove that
\begin{equation}\label{dfdfdf}
\forall k< j \leq [\frac{\omega}{2}] \quad (n-2)(d_j-d_k)+d_k\sqrt{\Delta_j}+d_j\sqrt{\Delta_k}>0
\end{equation}
Notice that  $\Delta_k$ given by \eqref{deltak} is a rational fraction in $n$. By straightforward computations, we check that there exists reel numbers $a_k,\;b_k,\;e_k,\;h_k$ and $s_k$  which depend on $k$ and $\omega$ such that
\begin{gather}
\Delta_k=a_kn^2+b_kn+e_k+\frac{h_k}{n-2}+\frac{s_k}{\nu_k+1-n}\\ 
\sqrt{\Delta_k}>\sqrt{a_k}(n+\frac{b_k}{2a_k})\label{efef}
\end{gather}

Inequality  \eqref{dfdfdf} holds if we use \eqref{efef}.\\
The expressions of the reel numbers above are known explicitly (we used the software Maple to compute them, see \cite{Mad3}). For simplicity, we omit to give these expressions.  
\end{itemize}
\end{proof}

\begin{proof}[\textbf{Proof of the main theorem}] The orbit of  $P$ under the action of $G$ is supposed to be minimal (i.e. $\Ca O_G(P)=\inf_{Q\in M}\Ca O_G(Q)$). Without loss of generality, we suppose that $3\leq \omega\leq (n-6)/2$, because if $\omega>(n-6)/2$ or $\omega\leq 2$, we conclude using Theorem~\ref{HV theorem}. From functions  $\tilde\varphi_{\e,i}$ defined by \eqref{varvar}, we define the function $\phi_\e$ as follows: 
$$\phi_\e=\sum_{k=1}^m\tilde\varphi_{\e,i}$$
$\phi_\e$ is $G-$invariant. In fact, for any $\sigma \in G$, such that $\sigma(P_i)=P_j$ 
$$u_{\e,P_i}=u_{\e,P_j}\circ\sigma\text{ and } f_i=f_j\circ \sigma$$  
$f_i$ are defined by \eqref{fitil}, we deduce that
$$\tilde\varphi_{\e,i}=\tilde\varphi_{\e,j}\circ\sigma$$
The support of $\tilde\varphi_{\e,i}$ is included in the ball $B_{P_i}(\delta)$. We choose $\delta$ sufficiently small such that for all integers $i\neq j$ in  $[ 1,m] $, intersection $B_{P_j}(\delta)\cap B_{P_i}(\delta)=\varnothing$. Thus
$$I_g(\phi_\e)=(\mathrm{card}O_G(P))^{2/n}I_g(\varphi_\e)$$
By Lemma \ref{lemme infg}, we conclude that 
$$I_g(\phi_\e)<\frac{n(n-2)}{4}\omega_{n-1}^{2/n}(\Ca O_G(P))^{2/n}$$
It remains to notice that if $\tilde g=\phi_\e^{4/(n-2)}g$ then
$$J(\tilde g)=4\frac{n-1}{n-2}I_g(\phi_\e)<n(n-1)\omega_{n-1}^{2/n}(\mathrm{card}O_G(P))^{2/n}$$
where $\e$ is sufficiently smaller than $\delta$.
\end{proof}

\begin{proof}[\textbf{Proof of the Corollary \ref{cor prin}}] 
Suppose that the orbit of  $P$ under the action of $G$ is minimal  (otherwise the conjecture is obvious). \\
If $\omega=\omega(P)> [(n-6)/2]$, we conclude using  Theorem \ref{HV theorem}.\\
If $\omega\leq [(n-6)/2] \leq 15$, we conclude using main theorem.
\end{proof}


\bibliographystyle{amsplain}

\end{document}